\documentclass[a4paper,reqno]{amsart}

\usepackage[utf8]{inputenc}
\usepackage[english]{babel}

\usepackage{amsfonts,amsthm,amssymb,mathtools}
\numberwithin{equation}{section}

\usepackage{hyperref}
\usepackage[capitalise,noabbrev]{cleveref} 
\crefname{equation}{}{}
\Crefname{equation}{}{}

\newtheorem{theorem}{Theorem}[section]
\newtheorem{lemma}[theorem]{Lemma}
\newtheorem{proposition}[theorem]{Proposition}

\theoremstyle{definition}

\theoremstyle{remark}
\newtheorem*{remark}{Remark}

\newcommand{\N}{\mathbb{N}}
\newcommand{\Z}{\mathbb{Z}}

\newcommand{\R}{\mathbb{R}}
\newcommand{\C}{\mathbb{C}}

\newcommand{\F}{\mathcal F}

\renewcommand{\S}{\mathcal{S}}

\DeclarePairedDelimiter\norm{\lVert}{\rVert}

\DeclareMathOperator{\supp}{supp}
\newcommand{\sloc}{\mathrm{sloc}}

\begin{document}

\title[Spectral multipliers for Grushin operators]
{A $p$-specific spectral multiplier theorem with sharp regularity bound for Grushin operators}
\author{Lars Niedorf}
\address{Mathematisches Seminar, Christian-Albrechts-Universität zu Kiel, Heinrich-Hecht-Platz 6, 24118 Kiel, Germany}
\email{niedorf@math.uni-kiel.de}
\date{June 22, 2022}
\thanks{This version of the article has been accepted for publication, after peer review but is not the Version of Record and does not reflect post-acceptance improvements, or any corrections. The Version of Record is available online at: \url{https://doi.org/10.1007/s00209-022-03029-0}.\\
The author gratefully acknowledges the support by the Deutsche Forschungsgemeinschaft (DFG) through grant MU 761/12-1.}

\maketitle

\begin{abstract}
In a recent work, P. Chen and E. M. Ouhabaz proved a $p$-specific $L^p$-spectral multiplier theorem for the Grushin operator acting on $\R^{d_1}\times\R^{d_2}$ which is given by
\[
L =-\sum_{j=1}^{d_1} \partial_{x_j}^2 - \bigg( \sum_{j=1}^{d_1} |x_j|^2\bigg) \sum_{k=1}^{d_2}\partial_{y_k}^2.
\]
Their approach yields an $L^p$-spectral multiplier theorem within the range $1< p\le \min\{ \frac{2d_1}{d_1+2},\frac{2(d_2+1)}{d_2+3} \}$ under a regularity condition on the multiplier which sharp only when $d_1\ge d_2$. In this paper, we improve on this result by proving $L^p$-boundedness under the expected sharp regularity condition $s>(d_1+d_2)(1/p-1/2)$. Our approach avoids the usage of weighted restriction type estimates which played a key role in the work of P. Chen and E. M. Ouhabaz, and is rather based on a careful analysis of the underlying sub-Riemannian geometry and restriction type estimates where the multiplier is truncated along the spectrum of the Laplacian on $\R^{d_2}$.
\end{abstract}

\section{Introduction}

Let $L$ be a positive self-adjoint linear differential operator on $L^2(M)$, where $M$ is a smooth $d$-dimensional manifold endowed with a smooth positive measure $\mu$. If $E$ denotes the spectral measure of $L$, we can define for every Borel measurable function $F:\R\to \C$ the (possibly unbounded) operator
\[
F(L) := \int_0^\infty F(\lambda) \,dE(\lambda).
\]
By the spectral theorem, $F(L)$ is a bounded operator on $L^2(M)$ if and only if the spectral multiplier $F$ is $E$-essentially bounded. The \textit{$L^p$-spectral multiplier problem} asks for identifying multipliers $F$ for which $F(L)$ extends from $L^2(M)\cap L^p(M)$ to a bounded operator $F(L):L^p(M)\to L^p(M)$.

For instance, in the case of the Laplacian $L=-\Delta$ on $\R^d$, the celebrated Mikhlin--Hörmander multiplier theorem \cite{Hoe60} provides the following sufficient condition for the question of $L^p$-boundedness: The operator $F(-\Delta)$ is bounded on $L^p(\R^d)$ for any $1<p<\infty$ whenever $F:\R\to\C$ satisfies the regularity condition
\[
\norm{F}_{\sloc,s} := \sup_{t>0}{ \norm{ \eta F(t\,\cdot\,)}_{L_s^2(\R)} } < \infty
\quad\text{for some } s>d/2.
\]
Here $\eta:\R\to\C$ shall denote some generic nonzero bump function supported in $(0,\infty)$, while $L_s^2(\R)\subseteq L^2(\R)$ is the Sobolev space of (fractional) order $s\in\R$. In the case $p=1$, the operator $F(-\Delta)$ is of weak type $(1,1)$, i.e., bounded as an operator between $L^1(\R^d)$ and the Lorentz space $L^{1,\infty}(\R^d)$. The threshold $d/2$ of the order $s$ is optimal and cannot be decreased.

A lot of attention has been paid to the question whether an analogous result of the Mikhlin--Hörmander multiplier theorem holds true for more general classes of (sub)-elliptic differential operators, most notably \textit{sub-Laplacians}. For left-invariant sub-Laplacians on Carnot groups, M.~Christ \cite{Ch91}, and G.~Mauceri and S.~Meda \cite{MaMe90} showed that $F(L)$ extends to a bounded operator on all $L^p$-spaces for $1<p<\infty$ and is of weak type $(1,1)$ whenever
\[
\norm{F}_{\sloc,s} < \infty
\quad\text{for some } s > Q/2,
\]
where $Q$ is the so-called \textit{homogeneous dimension} of the underlying Carnot group. It came therefore as a surprise when D.~Müller and E.~M.~Stein \cite{MueSt94}, and independently W.~Hebisch \cite{He93}, discovered in the early nineties that in the case of Heisenberg (-type) groups the threshold $s>Q/2$ can be even pushed down to $s>d/2$, with $d$ being the \textit{topological dimension} of the underlying group. The question whether this holds true for any sub-Laplacian $L$ is still open, although there has been extensive research on this problem and many partial results are available, including, e.g., sub-Laplacians on all 2-step nilpotent Lie groups of dimension $\le 7$ \cite{MaMue14b}, certain classes of 2-step nilpotent Lie groups of higher dimension \cite{Ma15}, Grushin operators \cite{MaMue14a}, as well as various classes of compact sub-Riemannian manifolds \cite{CoSi01,CoKlSi11,CaCoMaSi17,AhCoMaMue20}.
So far a counterexample requiring the threshold to be larger than $d/2$ is not known.

A refinement of asking for boundedness on all $L^p$-spaces for $1<p<\infty$ simultaneously is the question which order of differentiability $s$ is needed if $p$ is given (\textit{$p$-specific $L^p$-spectral multiplier estimates}). Again in the case of the Laplacian $L=-\Delta$, it is by now well-known (see \cite[Theorem 1.4]{LeRoSe14} for instance) that if $1< p \le 2(d+1)/(d+3)$ and if $F:[0,\infty)\to\C$ is a bounded Borel function satisfying
\[
\norm{F}_{\sloc,s} < \infty
\quad\text{for some } s > \max\{d\left|1/p-1/2\right|,1/2\},
\]
then the operator $F(-\Delta)$ is bounded on $L^p(\R^d)$. The condition on the range of $p$ derives from the celebrated Stein--Tomas Fourier restriction theorem \cite{To79} which is used for the proof of this result. It is an open problem (the famous \textit{Bochner--Riesz conjecture}, cf.\ \cite{So08, CaSj72, Ta04,BoGu11, Gu16}) in the case of Bochner--Riesz means (where $F=(1-|\cdot|)_+^\delta$, $\delta>0$) whether the operators $(1+\Delta)_+^\delta$ are bounded on $L^p(\R^d)$ whenever $\delta > \max\{d\left|1/p-1/2\right|-1/2,0\}$.

Regarding \textit{$p$-specific} $L^p$-spectral multiplier theorems for sub-Laplacians in more general settings, much fewer results featuring the topological dimension $d$ are available so far. However, in \cite{MaMueNi19}, A.~Martini, D.~Müller, and S.~Nicolussi Golo showed for a large class of smooth second-order real differential operators associated to a sub-Riemannian structure on smooth $d$-dimensional manifolds that regularity of order $s\ge d\left|1/p-1/2\right|$ is necessary for having $L^p$-spectral multiplier estimates. In particular, this result applies to all sub-Laplacians on Carnot groups, and Grushin operators, which are the subject of the present paper.

Quite recently, P.~Chen and E.~M.~Ouhabaz \cite{ChOu16} proved a partial result for a $p$-specific $L^p$-spectral multiplier estimate in the case of the Grushin operator $L$ acting on $\R^d=\R^{d_1}\times\R^{d_2}$, $d_1,d_2\ge 1$, given by
\[
L
= -\sum_{j=1}^{d_1} \partial_{x_j}^2 - \bigg( \sum_{j=1}^{d_1} |x_j|^2\bigg) \sum_{k=1}^{d_2}\partial_{y_k}^2
= -\Delta_{x} - |x|^2\Delta_{y} .
\]
Here $x\in\R^{d_1}$, $y\in\R^{d_2}$ shall denote the two layers of a given point in $\R^d$, while $\Delta_{x},\Delta_{y}$ are the corresponding partial Laplacians, and $|x|$ is the Euclidean norm of $x$. The Grushin operator is positive, self-adjoint, and hypoelliptic according to a celebrated theorem by Hörmander \cite{Hoe67}, but not elliptic on the plane $x=0$. In \cite{ChOu16}, it is proved that $F(L)$ extends to a bounded operator on $L^p(\R^d)$ whenever
\[
\norm{F}_{\sloc,s} < \infty
\quad\text{for some } s > D (1/p - 1/2 ),
\]
where $D:=\max\{d_1+d_2,2d_2\}$ and $1< p\le p_{d_1,d_2}$, with
\begin{equation}\label{eq:p}
p_{d_1,d_2}:=\min\Big\{ \frac{2d_1}{d_1+2},\frac{2(d_2+1)}{d_2+3} \Big\}.
\end{equation}
As suspected by P.\ Chen and E.\ M.\ Ouhabaz in \cite{ChOu16}, one might expect that this result holds true with $D$ being replaced by the topological dimension $d=d_1+d_2$. However, their result yields the optimal threshold at least if $d_1\ge d_2$.

A similar phenomenon as in \cite{ChOu16} had already occurred earlier in \cite{MaSi12}, where A.~Martini and A.~Sikora proved a Mikhlin--Hörmander type result for the Grushin operator $L$ with threshold $s>D/2$, which was later improved in \cite{MaMue14a} by A.~Martini and D.~Müller to hold for the topological dimension $d$ in place of $D$. The approaches of \cite{MaSi12} and \cite{MaMue14a} rely both on weighted Plancherel estimates for the integral kernels of $F(L)$, which are derived by pointwise estimates for Hermite functions. In \cite{MaSi12}, the employed weights are given by $w_\gamma(x,y)=|x|^\gamma$, $\gamma>0$. In principle, the arguments work out for $\gamma<d_2/2$, but unfortunately, it is necessary to take an integral over the weight $|x|^\gamma$ at some point, which forces $\gamma<d_1/2$, which in turn yields $s>D/2$ in place of $s>d/2$ as a threshold. In \cite{MaMue14a}, A.~Martini and D.~Müller employ the weights $w_\gamma(x,y)=|y|^\gamma$ in the second layer $y$, together with a rescaling factor in the first layer. Using the weights $|y|^\gamma$ does only force $\gamma<d_2/2$ when taking the integral over the weight, whence this approach provides the optimal threshold $s>d/2$. However, the weights $|y|^\gamma$ are harder to handle since a sub-elliptic estimate, which goes back to W.\ Hebisch \cite{He93}, is not applicable for these weights.

The proof of P.\ Chen and E.\ M.\ Ouhabaz relies on weighted restriction type estimates using $|x|^\gamma$ as a weight. Similar to \cite{MaSi12}, they employ Hebisch's sub-elliptic estimate and have to take an integral over the weight $|x|^\gamma$ which forces $\gamma<d_1(1/p-1/2)$, and in turn yields $s>D(1/p-1/2)$ in place of $s>d(1/p-1/2)$ as a threshold.

In this paper, we improve the result of \cite{ChOu16} and prove a $p$-specific spectral multiplier estimate with optimal threshold for $s$. Similar as in \cite{ChOu16}, we also prove a corresponding result for Bochner--Riesz multipliers. Note that \cref{thm:Lp-multiplier} only provides results if $d_1\ge 3$ and $d_2\ge 2$, and \cref{thm:riesz-means} if $d_1\ge 2$ and $d_2\ge 1$.

\begin{theorem}\label{thm:Lp-multiplier}
Let $1< p\le p_{d_1,d_2}$. Suppose that $F:\R\to \C$ is a bounded Borel function such that
\[
\norm{F}_{\sloc,s} < \infty
\quad\text{for some } s > (d_1+d_2)(1/p - 1/2).
\]
Then the operator $F(L)$ is bounded on $L^p(\R^d)$, and
\[
\norm{F(L)}_{L^p\to L^p} \le C_{p,s} \norm{F}_{\sloc,s}.
\]
\end{theorem}

\begin{theorem}\label{thm:riesz-means}
Let $1\le p\le p_{d_1,d_2}$. Suppose that $\delta> (d_1+d_2)( 1/ p -  1/ 2)- 1/2$. Then the Bochner--Riesz means $(1-tL)^\delta_+$ are bounded on $L^p(\R^d)$ uniformly in $t\in[0,\infty)$.
\end{theorem}

Our strategy when reaching for the optimal threshold $s>d(1/p-1/2)$ is to follow the approach by P.~Chen and E.~M.~Ouhabaz, but instead of showing \textit{weighted} restriction type estimates, we prove restriction type estimates where the operator $F(L)$ is additionally truncated along the spectrum of the Laplacian on $\R^{d_2}$. On a heuristic level, this key idea may be illustrated as follows: Via Fourier transform in the second component, the study of the operator $L$ translates into studying the family of operators $-\Delta_x+|x|^2|\eta|^2$, $\eta\in\R^{d_2}$, on $L^2(\R^{d_1})$. For fixed $\eta\in\R^{d_2}$, this operator is a rescaled version of the Hermite operator, and has discrete spectrum consisting of the eigenvalues $[k]|\eta|$, where $[k]=2k+d_1$ and $k\in\N$. Moreover, the operator $T:=(-\Delta_y)^{1/2}$ translates into the multiplication operator $|\eta|$ via  Fourier transform in the second component. The operators $L$ and $T$ admit a joint functional calculus, and since $[k]|\eta|/|\eta|=[k]$, multiplication with the operator $\chi_k(L/T)$ (where $\chi_k:\R\to\C$ shall denote the indicator function of $\{2k+d_1\}$) corresponds to picking the $k$-th eigenvalue on $L^2(\R^{d_1})$ for every $\eta\in\R^{d_2}$ simultaneously. This is an observation that has been already been exploited earlier, for instance in \cite[Lemma 11]{MaMue14a}, and in \cite{MueRiSt95,MueRiSt96}. Since $r \sim [k]^{-1}$ on the support of a joint multiplier $F(\lambda)\chi_k(\lambda/r)$ whenever $F$ is supported away from the origin, the multiplication of an operator $F(L)$ by $\chi_k(L/T)$ is referred to as a \textit{truncation along the spectrum of $T$} in the following. The benefit of this truncation is as follows: Since $L$ and $T$ admit a joint functional calculus, we have
\[
F(L) \chi_k(L/T) = F([k] T) \chi_k(L/T).
\]
Thus for every $k\in\N$, we may replace the operator $L$ by the Laplacian in the second layer $y\in\R^{d_2}$, whence one might hope that on each “eigenspace” associated to $k$ the underlying sub-Riemannian geometry behaves Euclidean up to a scaling by $k$ in the second layer. In the proofs of \cref{thm:Lp-multiplier} and \cref{thm:riesz-means}, we will take advantage of this perspective in the case where $k\in\N$ is small.

\medskip

This article is organized as follows: In \cref{sec:cc}, we recall the main facts concerning the sub-Riemannian geometry that is naturally associated to the Grushin operator $L$. In \cref{sec:rest}, we recall the essentials of the joint functional calculus of $L$ and $T$ and prove the truncated restriction type estimates mentioned above. \Cref{sec:proof} is devoted to the proof of \cref{thm:Lp-multiplier} and \cref{thm:riesz-means}, where also a closer analysis of the underlying sub-Riemannian geometry takes place.

\medskip

Finally, we briefly fix our notation. For us, zero shall be contained in the set of all natural numbers $\N$. The space of (equivalence classes of) integrable simple functions on $\R^n$ will be denoted by $D(\R^n)$, while $\S(\R^n)$ shall denote the space of Schwartz functions on $\R^n$. The indicator function of a subset $A\subseteq \R^n$ will be denoted by $\chi_A$. For a function $f\in L^1(\R^n)$, the Fourier transform $\hat f$ is given by
\[
\hat f(\xi) = \int_{\R^n} f(x) e^{-i\xi x}\, dx,\quad
\xi\in\R^n,
\]
while the inverse Fourier transform $\check f$ is given by
\[
\check f(x) = (2\pi)^{-n} \int_{\R^n} f(\xi) e^{ix\xi}\, d\xi,\quad
x\in\R^n.
\]
Constants may vary from line to line, but they will be occasionally denoted by the same letter. We write $A\lesssim B$ if $A\le C B$ for a constant $C$. If $A\lesssim B$ and $B\lesssim A$, we write $A\sim B$. Moreover, we fix the following dyadic decomposition throughout this article. Let $\chi:\R\to [0,1]$ be an even bump function supported in $[-2,-1/2]\cup [1/2,2]$ such that
\[
\sum_{j\in\Z} \chi_j(\lambda) = 1 \quad \text{for } \lambda\neq 0,
\]
where $\chi_j$ is given by
\begin{equation}\label{eq:dyadic}
\chi_j(\lambda):=\chi(\lambda/2^j) \quad \text{for }j\in\Z.
\end{equation}
With this setup, we have in particular $|\lambda|\sim 2^j$ for all $\lambda\in\supp\chi_j$.

\medskip

\thanks{\textsc{Acknowledgment.} I would like to express my gratitude to my advisor Professor Dr. Detlef Müller for his constant support and numberless helpful suggestions.}

\section{The sub-Riemannian geometry of the Grushin operator}
\label{sec:cc}

Let $\varrho$ denote the Carnot--Carathéodory distance associated to the Grushin operator $L$, i.e., for $z,w\in \R^d$, the distance $\varrho(z,w)$ is given by the infimum over all lengths of horizontal curves $\gamma:[0,1]\to\R^d$ joining $z$ with $w$ (cf.\ Section III.4 of \cite{VaSaCo92}). Due to the Chow--Rashevkii theorem (cf.\ Proposition III.4.1 in \cite{VaSaCo92}), $\varrho$ is indeed a metric on $\R^d$, which induces the Euclidean topology on $\R^d$. In our setting, the Carnot--Carathéodory distance possesses the following characterization (cf.\ Proposition 3.1 in \cite{JeSa87}): If $z,w\in\R^d$, then
\begin{equation}\label{eq:cc-char}
\varrho(z,w) = \sup_{\psi\in \Lambda} (\psi(z)-\psi(w)),
\end{equation}
where $\Lambda$ denotes the set of all locally Lipschitz continuous functions $\psi:\R^d\to \R$ such that
\begin{equation}\label{eq:cc-domain-grushin}
|\nabla_x\psi(x,y)|^2 + |x|^2 |\nabla_y\psi(x,y)|^2 \le 1
\quad \text{for a.e. } (x,y)\in \R^{d_1}\times\R^{d_2}.
\end{equation}
In the following, let $B_R^{\varrho}(a,b)$ denote the ball of radius $R\ge 0$ centered at $(a,b)\in \R^{d_1}\times\R^{d_2}$ with respect to the distance $\varrho$. The following statement summarizes the main properties of the sub-Riemannian geometry associated to $L$ that we need later.

\begin{proposition}\label{prop:cc-distance}
The following statements hold:
\begin{enumerate}
\item For all $(x,y),(a,b)\in \R^{d_1}\times \R^{d_2}$,
\[
\varrho((x,y),(a,b)) \sim |x-a| +
\begin{cases}
\frac{|y-b|}{|x|+|a|} & \text{if } |y-b|^{1/2}< |x| + |a|, \\
|y-b|^{1/2} & \text{if }  |y-b|^{1/2}\ge |x| + |a| .
\end{cases}
\]
\item For all $(a,b)\in \R^{d_1}\times \R^{d_2}$ and $R>0$,
\[
|B_R^{\varrho}(a,b)| \sim R^{d_1+d_2} \max\{R,|a|\}^{d_2}.    
\]
\item There is a constant $C>0$ such that for all $(a,b)\in \R^{d_1}\times \R^{d_2}$,
\[
B_R^{\varrho}(a,b)
 \subseteq B_R^{|\,\cdot\,|}(a) \times B_{ CR^2}^{|\,\cdot\,|}(b)
  \quad \text{whenever } R\ge |a|/4.
\]
\item Let $\delta_t(x,y):=(tx,t^2y)$ for $(x,y)\in \R^{d_1}\times \R^{d_2}$. Then
\[
\varrho(\delta_t(x,y),\delta_t(a,b)) = t\varrho((x,y),(a,b)).
\]
\item $L$ possesses the \textit{finite propagation speed property} with respect to $\varrho$, i.e., whenever $f,g\in L^2(\R^d)$ are supported in open subsets $U,V\subseteq \R^d$ and $|t|< \varrho(U,V)$, then
\[
(\cos(t\sqrt L)f,g) = 0.
\]
\end{enumerate}
\end{proposition}

\begin{proof}
The estimates in (1) and (2) are part of Proposition 5.1 in \cite{RoSi08}. The inclusion in (3) is a consequence of (1): Since the function $\psi$ defined by $\psi(x,y):=|x-a|$ satisfies \cref{eq:cc-domain-grushin}, the characterization \cref{eq:cc-char} yields
\[
|x-a|\le \varrho((x,y),(a,b)).
\]
Thus, if we suppose $(x,y)\in B_R^{\varrho}(a,b)$ for $R\ge |a|/4$, then the inequality above implies $x\in B_R^{|\cdot|}(a)$, and $|x|\le |x-a|+|a|< 5R$. Moreover, if $|y-b|^{1/2}< |x| + |a|$, then (1) yields
\begin{align*}
|y-b|
 \lesssim (|x|+|a|) \varrho((x,y),(a,b)) 
 < 9R^2,
\end{align*}
and if $|y-b|^{1/2}\ge |x| + |a|$, then $|y-b| \lesssim \varrho((x,y),(a,b))^2 < R^2$, which proves (3). The scaling invariance in (4) is an immediate consequence of the characterization \cref{eq:cc-char}. For the finite propagation speed property, see Proposition 4.1 of \cite{RoSi08}, or alternatively the approach of Melrose in \cite[Proposition 3.4]{Me84}.
\end{proof}

The finite propagation speed property will be of fundamental importance in the proofs of \cref{thm:Lp-multiplier} and \cref{thm:riesz-means}. Moreover, note that the volume estimate in part (2) of \cref{prop:cc-distance} yields in particular that the metric measure space $(\R^d,\varrho,|\cdot|)$ (with $|\cdot|$ denoting the Lebesgue measure) is a space of homogeneous type with homogeneous dimension $Q = d_1+2d_2$.

\section{Truncated restriction type estimates}
\label{sec:rest}

In this section, we prove restriction type estimates where the multiplier is additionally truncated along the spectrum of the Laplacian on $\R^{d_2}$. As in \cite{ChOu16}, the idea is to apply a discrete restriction estimate in the variable $x\in\R^{d_1}$ and the classical Stein--Tomas restriction estimate in $y\in\R^{d_2}$. Due to the conditions $1\le p\le 2d_1/(d_1+2)$ and $1\le p \le 2(d_2+1)/(d_2+3)$ in the corresponding restriction type estimates, we have to assume $1\le p\le p_{d_1,d_2}$ in \cref{thm:rest} (with $p_{d_1,d_2}$ being defined as in \cref{eq:p}).

We first discuss the spectral decomposition of the Grushin operator $L$. Let $\F_2 : L^2(\R^d) \to L^2(\R^d)$ denote the Fourier transform in the second component, i.e.,
\[
\F_2 f(x,\eta) = \int_{\R^{d_2}} f(x,y) e^{-i\eta y} \,dy, \quad f\in\S(\R^d).
\]
We will also write $f^\eta(x)=\F_2 f(x,\eta)$ in the following. Then
\[
(Lf)^\eta = (-\Delta_x + |x|^2|\eta|^2)f^\eta.
\]
For fixed $\eta\in\R^{d_2}$, the operator
\[
L^\eta := -\Delta_x + |x|^2|\eta|^2 \quad \text{on } L^2(\R^{d_1})
\]
is a rescaled version of the Hermite operator $H=-\Delta+|x|^2$ on $\R^{d_1}$. It is well-known \cite[Section 1.1]{Th93} that $H$ has discrete spectrum consisting of the eigenvalues
\[
[k]:=2k+d_1,\quad k\in\N.
\]
For a multiindex $\nu\in\N^{d_1}$, let $\Phi_\nu$ denote the $\nu$-th Hermite function on $\R^{d_1}$, i.e.,
\[
\Phi_\nu(x) := \prod_{j=1}^{d_1} h_{\nu_j}(x_j),\quad x\in \R^{d_1},
\]
where, for $\ell\in\N$, $h_\ell$ shall denote the $\ell$-th Hermite function on $\R$, i.e.,
\[
h_\ell(u) := (-1)^\ell (2^\ell \ell! \sqrt{\pi})^{-1/2} e^{u^2/2} \Big(\frac{d}{du}\Big)^\ell (e^{-u^2}),\quad u\in \R.
\]
The Hermite functions $\Phi_\nu$ form an orthonormal basis of $L^2(\R^{d_1})$ and are eigenfunctions of the Hermite operator $H$ since $H\Phi_\nu  = (2|\nu|_1+d_1)  \Phi_\nu$, where $|\nu|_1 =\nu_1+\ldots+\nu_{d_1}$ denotes the length of the multiindex $\nu\in\N^{d_1}$.

Furthermore, for $\eta\in\R^{d_2}$, let $\Phi_\nu^\eta$ be given by 
\[
\Phi_\nu^\eta(x) := |\eta|^{d_1/4} \Phi_\nu(|\eta|^{1/2}x),\quad x\in \R^{d_1}.
\]
Then the functions $\Phi_\nu^\eta$ form an orthonormal basis of $L^2(\R^{d_1})$ and are eigenfunctions of the operator $L^\eta$ since $L^\eta\Phi_\nu^\eta = (2|\nu|_1+d_1)|\eta| \Phi_\nu^\eta$. Thus the projection $P_k^\eta$ onto the eigenspace associated to the eigenvalue $[k]|\eta|$ of $L^\eta$ is given by
\[
P_k^\eta g = \sum_{|\nu|_1=k} ( g, \Phi_{\nu}^\eta )  \Phi_{\nu}^\eta,
\quad
g\in L^2(\R^{d_1}).
\]
In particular, the projection $P_k^\eta$ possesses an integral kernel $\mathcal K_k^\eta$ which is given by
\begin{equation}\label{eq:integral-kernel}
\mathcal K_k^\eta(x,a) = \sum_{|\nu|_1=k} \Phi_\nu^\eta(x) \Phi_\nu^\eta(a),
\quad x,a\in\R^{d_1}.
\end{equation}
Moreover, let $L_j$ and $T_k$ be the differential operators given by 
\[
L_j = (-i\partial_{x_j})^2 + |x_j|^2 \sum_{k=1}^{d_2}(-i\partial_{y_k})^2,
\qquad
T_k = -i \partial_{y_k}.
\]
Then the Grushin operator $L$ is equal to the sum $L_1+\dots+L_{d_1}$. As shown in \cite{MaSi12}, the operators $L_1,\dots,L_{d_1},T_1,\dots,T_{d_2}$ have a joint functional calculus which can be explicitly written down in terms of the Fourier transform and Hermite function expansion. In particular, the operators $L$ and $T=(|T_1|^2+\dots+|T_{d_2}|^2)^{1/2}=(-\Delta_y)^{1/2}$ have a joint functional calculus, so we can define the operators $G(L,T)$ for every Borel function $G:\R\times \R \to \C$.

\begin{lemma}\label{lem:func-calc}
For all bounded Borel functions $G:\R\times \R \to \C$,
\begin{equation*}
( G(L,T) f )^\eta(x) = G(L^\eta , |\eta|) f^\eta(x)
\end{equation*}
for all $f\in L^2(\R^d)$ and almost all $(x,\eta)\in \R^{d_1}\times\R^{d_2}$. Moreover, if $G$ is additionally compactly supported in $\R\times (\R\setminus\{0\})$, the operator $G(L,T)$ possesses an integral kernel $\mathcal K_{G(L,T)}$, which is given by
\[
\mathcal K_{G(L,T)}((x,y),(a,b)) = (2\pi)^{-d_2} \int_{\R^{d_2}} \sum_{k=0}^\infty G\big([k]|\eta|, |\eta|\big) \mathcal K_k^\eta(x,a) e^{i(y-b)\eta}\, d\eta
\]
for almost all $(x,y),(a,b)\in\R^{d_1}\times\R^{d_2}$.
\end{lemma}

\begin{proof}
See Proposition 5 of \cite{MaSi12}, and its proof.
\end{proof}

The proof of the truncated restriction type estimates for the Grushin operator relies on the following restriction type estimate for $L^\eta$.

\begin{proposition}\label{thm:discrete-rest}
Let $1\le p \le 2d_1/(d_1+2)$. Then, for all $g\in D(\R^{d_1})$ and $\eta\in\R^{d_2}$,
\begin{equation*}
\norm{P_k^\eta g}
 \le C_{p} |\eta|^{\frac{d_1}{2}( 1/p - 1/2)} [k]^{\frac {d_1} 2 ( 1/p - 1/2) -1/2} \norm{g}_p.
\end{equation*}
\end{proposition}

\begin{proof}
Via substitution, the proof of the estimate can be reduced to the case where $|\eta|=1$ (cf.\ Proposition 3.2 of \cite{ChOu16}). For the case $|\eta|$ = 1, see Corollary 3.2 of \cite{KoTa05}. Alternatively, for $1\le p< 2d_1/(d_1+2)$, this result can also be found in \cite[Theorem 3]{Ka95} and \cite[Proposition II.8]{ChOuSiYa16} (in conjunction with Mehler's formula).
\end{proof}

Another ingredient for the proof of the restriction type estimates are pointwise estimates for Hermite functions. In the following, we let
\begin{equation}\label{eq:Hermite-kernel}
H_k^\eta(x):=\mathcal K_k^\eta(x,x),\quad  x \in\R^{d_1}.
\end{equation}

\begin{lemma}\label{thm:hermite}
If $d_1\ge 2$, then, for all $k\in\N$ and $\eta\in\R^{d_2}$,
\[
H_k^\eta(x) \leq
\begin{cases}
C |\eta|^{d_1/2} [k]^{d_1 / 2-1} & \text { for all } x \in \mathbb{R}^{d_1}, \\
C  |\eta|^{d_1/2} \exp (-c|\eta||x|_{\infty}^{2}) & \text { when }|\eta||x|_{\infty}^{2} \geq 2[k].
\end{cases}
\]
\end{lemma}

\begin{proof}
See \cite[Lemma 8]{MaSi12} and the references therein.
\end{proof}

Now we state the restriction type estimates of the Grushin operator $L$. The new feature in comparison to \cite{ChOu16} is the truncation along the spectrum of $T$ instead of employing weights in the restriction type estimates. Let $\varrho$ denote again the Carnot--Carathéodory distance associated to $L$.

\begin{theorem}\label{thm:rest}
Let $1\le p\le p_{d_1,d_2}$. Suppose that $F:\R\to\C$ is a bounded Borel function supported in $[1/8,8]$. For $\ell\in\N$, let $G_\ell : \R\times \R \to \C$ be given by
\[
G_\ell(\lambda,r) = F(\sqrt \lambda)\chi_\ell(\lambda/r) \quad \text{for }r\neq 0
\]
and $G_\ell(\lambda,r) = 0$ else, where $\chi_\ell$ is defined via \cref{eq:dyadic}. Then
\begin{align}\label{eq:rest-1}
\norm{ G_\ell(L,T) }_{p\to 2}
&\le C_{p} 2^{-\ell d_2(1/p -1/2)} \norm{F}_2.
\end{align}
In particular, for $\iota\in\N$, 
\begin{equation}\label{eq:rest-2}
\bigg\Vert \sum_{\ell > \iota } G_\ell(L,T) \bigg\Vert_{p\to 2} \le C_{p} 2^{-\iota d_2(1/p -1/2)} \norm{F}_2.
\end{equation}
Moreover, for $(a,b)\in \R^{d_1}\times\R^{d_2}$ and $0<R<|a|/4$,
\begin{equation}\label{eq:rest-3}
\norm{F( \sqrt L) \chi_{B_R^\varrho(a,b)}}_{p\to 2}
\le C_p |a|^{-d_2(1/ p -  1/ 2)} \norm{F}_2.
\end{equation}
\end{theorem}

\begin{remark}
By \cref{lem:func-calc}, we have
\[
(G_\ell(L,T) f)^\eta = \sum_{k=0}^\infty F(\sqrt{[k]|\eta|}) \chi_\ell ([k] ) P_k^\eta f^\eta
\]
for almost all $\eta\in\R^{d_2}$. Note that $d_1\ge 2$ due to the assumption on the range of $p$. Thus $\chi_j([k])=0$ for all $j\le 0$ and $k\in\N$, whence
\[
\sum_{\ell=1}^\infty G_\ell(L,T) f = F(\sqrt L)f.
\]
\end{remark}

\begin{proof}
We first prove \cref{eq:rest-1}. Note that \cref{eq:rest-2} is a direct consequence of \cref{eq:rest-1} since
\[
\bigg\Vert \sum_{\ell> \iota} G_\ell(L,T) \bigg\Vert_{p\to 2}
 \le \sum_{\ell> \iota} \Vert G_\ell(L,T)\Vert_{p\to 2}.
\]
Let $f\in \S(\R^d)$. In the following, let $g_k^\eta := F( \sqrt{[k]|\eta|}) f^\eta$ for $\eta\in\R^{d_2}$ and $k\in\N$. Using Plancherel's theorem, \cref{lem:func-calc}, and orthogonality in $L^2(\R^{d_1})$, we obtain
\begin{align}
\norm{ G_\ell(L,T) f }^2_{L^2(\R^d)}
& \sim \norm{ G_\ell(L^\eta,|\eta|) f^\eta } ^2_{L^2(\R^{d_1}\times\R^{d_2}_{\eta})} \notag \\
& = \bigg \Vert \sum_{k=0}^\infty F(\sqrt{[k]|\eta|}) \chi_\ell ([k] ) P_k^\eta f^\eta \bigg\Vert^2_{L^2(\R^{d_1}\times\R^{d_2}_{\eta})} \notag \\
& = \sum_{k=0}^\infty \chi_\ell([k])^2 \norm{ P_k^\eta g_k^\eta }^2_{L^2(\R^{d_1}\times\R^{d_2}_{\eta})}. \label{eq:prf-rest-1}
\end{align}
The restriction type estimate of \cref{thm:discrete-rest} provides the estimate
\begin{align}
\norm{ P_k^\eta g_k^\eta }_{L^2(\R^{d_1})}
& \lesssim |\eta|^{\frac{d_1}{2}( 1/p - 1/2)} [k]^{\frac {d_1} 2( 1/p - 1/2) -1/2} \norm{ g_k^\eta }_{L^p(\R^{d_1})} \notag \\
& \sim [k]^{-1/2} \norm{ g_k^\eta}_{L^p(\R^{d_1})}
\label{eq:appl-discrete-restriction}
\end{align}
since $[k] |\eta|\sim 1$ whenever $[k] |\eta|\in \supp F$. Moreover, Minkowski's integral inequality yields
\begin{equation}\label{eq:prf-rest-minkowski}
\norm{\norm{g_k^\eta}_{L^p(\R^{d_1})}}_{L^2(\R^{d_2}_{\eta})}
 \le \norm{\norm{g_k^\eta(x)}_{L^2(\R_\eta^{d_2})}}_{L^p(\R^{d_1}_{x})}.
\end{equation}
Let $f_{x}:=f(x,\cdot)$ and $\widehat\cdot$ denote the Fourier transform on $\R^{d_2}$. Using polar coordinates and applying the classical Stein--Tomas restriction estimate \cite{To79} yields
\begin{align}
\norm{ g_k^\eta(x) }_{L^2(\R_\eta^{d_2})}^2
 & = \int_0^\infty \int_{S^{d_2-1}} | F(\sqrt{[k]r}) \widehat{f_{x}}(r\omega)|^2 r^{d_2-1} \, d\sigma(\omega) \, dr\notag\\
 & = \int_0^\infty | F(\sqrt{[k]r})|^2 r^{-d_2-1}\int_{S^{d_2-1}} \big| \big(f_{x}(r^{-1}\,\cdot\,)\big)^\wedge(\omega)\big|^2 \, d\sigma(\omega) \, dr \notag \\
 & \lesssim \int_0^\infty | F(\sqrt{[k]r})|^2  r^{-d_2-1} \norm{f_{x}(r^{-1}\,\cdot\,)}_{L^p(\R^{d_2})}^2 \, dr\notag \\
 & = \int_0^\infty |F(\sqrt{[k]r})|^2 r^{2d_2(1/p -1/2)-1}  \, dr \, \norm{f_{x}}_{L^p(\R^{d_2})}^2\notag \\
 & \sim [k]^{-2d_2(1/p -1/2)} \int_0^\infty |F(\sqrt{r})|^2 \,dr \, \norm{f_{x}}_{L^p(\R^{d_2})}^2.\notag 
\end{align}
Substituting $r\mapsto r^2$, we obtain, together with \cref{eq:prf-rest-minkowski},
\begin{equation}\label{eq:prf-rest-5}
\norm{\norm{g_k^\eta}_{L^p(\R^{d_1})}}_{L^2(\R^{d_2}_{\eta})}
 \lesssim [k]^{-d_2(1/p -1/2)} \norm{F}_2 \norm{f}_p.
\end{equation}
Together with \cref{eq:appl-discrete-restriction}, we get
\begin{equation}\label{eq:prf-rest-4}
\norm{ P_k^\eta g_k^\eta  }_{L^2(\R^{d_1}\times\R^{d_2}_{\eta})}
\lesssim [k]^{-d_2(1/p -1/2)-1/2} \norm{F}_2 \norm{f}_p.
\end{equation}
Hence, in conjunction with \cref{eq:prf-rest-1}, we finally get
\begin{align*}
\norm{ G_\ell(L,T) f }^2_{L^2(\R^d)}
& \lesssim \sum_{[k]\sim 2^\ell} [k]^{-2d_2(1/p -1/2)-1} \norm{F}_2^2 \norm{f}_p^2 \\
& \lesssim 2^{-2\ell d_2(1/p -1/2)} \norm{F}_2^2 \norm{f}_p^2.
\end{align*}
This proves \cref{eq:rest-1}.

\medskip

Now we prove \cref{eq:rest-3}. Suppose that $f$ is supported in $B_R^\varrho(a,b)$. Applying \cref{eq:rest-2} for $\iota = 0$, we obtain
\[
\norm{F(\sqrt L) f}_2
\lesssim \norm{F}_2 \norm{f}_p.
\]
Hence we can assume $|a|>1$ without loss of generality. As before, let $g_k^\eta = F(\sqrt{[k] |\eta|}) f^\eta$. The same arguments as in \cref{eq:prf-rest-1} show that
\[
\Vert F(\sqrt L) f \Vert ^2_{L^2(\R^d)}
\sim \sum_{k=0}^\infty \Vert P_k^\eta g_k^\eta  \Vert^2_{L^2(\R^{d_1}\times\R^{d_2}_{\eta})}.
\]
We split the sum over $k$ in two parts, one part where $[k] \ge \gamma |a|$, and another part where $[k] < \gamma |a|$. The constant $\gamma>0$ will be chosen later sufficiently small.

\medskip

\textit{Case 1:} For those $k\in\N$ satisfying $[k] \ge \gamma |a|$, we use estimate \cref{eq:prf-rest-4} from before, and we are done since
\[
\sum_{[k]\ge \gamma |a|}[k]^{-2d_2(1/p -1/2)-1}
 \lesssim_\gamma |a|^{-2d_2(1/p -1/2)}.
\]

\medskip

\textit{Case 2:} For $[k] < \gamma |a|$, we replace the restriction type estimate of \cref{thm:discrete-rest} by an estimation that uses Hölder's inequality and the pointwise estimates for Hermite functions provided by \cref{thm:hermite}. (Note that we have assumed $d_1\ge 2$ by choosing $1\le p\le p_{d_1,d_2}$.) For the component $y\in\R^{d_2}$, we use the Stein--Tomas restriction estimate in the same way as before.

Fix $k\in\N$ with $[k] < \gamma |a|$. By \cref{prop:cc-distance} (3), $g_k^\eta$ is supported in $B_R^{|\,\cdot\,|}(a)$ since $f$ is supported in $B_R^\varrho(a,b)$.
Recall that the projection $P_k^\eta$ onto the eigenspace associated to the eigenvalue $[k]|\eta|$ possesses the integral kernel $\mathcal K_k^\eta$ given by \cref{eq:integral-kernel}. Using Hölder's inequality, we obtain
\begin{align*}
|P_k^\eta g_k^\eta (x)|
& \le \norm{ \mathcal K_k^\eta (x,\cdot) }_{L^{p'}(B_R^{|\,\cdot\,|}(a))}\norm{g_k^\eta}_{L^p(\R^{d_1})}, 
\end{align*}
where $p'$ is the dual exponent of $p$. Hence
\begin{equation}
\norm{ P_k^\eta g_k^\eta }_{L^2(\R^{d_1}\times\R_{\eta}^{d_2})}
\le \big \Vert  \underbrace{\norm{ \norm{ \mathcal K_k^\eta (x,\cdot) }_{L^{p'}(B_R^{|\,\cdot\,|}(a))} }_{L^2(\R_{x}^{d_1})} }_{=:\beta_{k,\eta}} \norm{g_k^\eta}_{L^p(\R^{d_1})} \big\Vert_{L^2(\R_\eta^{d_2})}.\label{eq:prf-rest-hoelder}
\end{equation}
Since $|\mathcal K_k^\eta (x,\xi)|\le H_k^\eta (x)^{1/2} H_k^\eta (\xi)^{1/2}$ (with $H_k^\eta$ being defined as in \cref{eq:Hermite-kernel}), we get
\begin{align}
\beta_{k,\eta}
& \le  \norm{ (H_k^\eta)^{1/2} }_{L^{2}(\R^{d_1})} \norm{ (H_k^\eta)^{1/2} }_{L^{p'}(B_R^{|\,\cdot\,|}(a))}.\label{eq:prf-rest-hermite}
\end{align}
The first factor can be estimated by
\begin{align}
\norm{ (H_k^\eta)^{1/2} }_{L^{2}(\R^{d_1})}
& = \bigg( \sum_{|\nu|_1=k} \Vert \Phi_\nu^\eta \Vert_2^2 \bigg)^{1/2}\notag \\
& = |\{\nu \in\N^{d_1} :|\nu|_1=k \}|^{1/2} \le k^{d_1/2}.\label{eq:hermite-kernel-L2}
\end{align}
Let $x\in B_R(a)$. Since $P_k^\eta g_k^\eta=0$ for $[k]|\eta|\notin \supp F$, we may assume $[k]|\eta| \sim 1$. Thus, since $R<|a|/4$, we have
\[
|\eta||x|_\infty^2
 \sim |\eta||x|^2
 \gtrsim |\eta| |a|^2  
 \ge \frac{|\eta|[k]^2 }{\gamma^2}
 \sim \frac{[k]}{\gamma^2}.
\]
Choosing $\gamma>0$ small enough absorbs all constants, so that $|\eta| |x|_\infty^2 \ge 2[k]$. Thus, together with \cref{thm:hermite}, we obtain
\begin{align}
\norm{ (H_k^\eta)^{1/2} }_{L^{p'}(B_R^{|\,\cdot\,|}(a))}
& \lesssim \big\Vert |\eta|^{d_1/4} \exp(- c |\eta||\cdot|^{2}) \big\Vert_{L^{p'}(B_R^{|\,\cdot\,|}(a))} \notag \\
& \le |\eta|^{d_1/4} \exp(-\tilde c |\eta| |a|^{2})  |B_R^{|\,\cdot\,|}(a)|^{1/p'}. \notag
\end{align}
Recall that we have assumed $|a|>1$, whence
\[
|\eta|^{d_1/4}|B_R^{|\,\cdot\,|}(a)|^{1/p'}\lesssim (|\eta||a|^2)^{d_1/4}.
\]
Moreover, since $[k]|\eta|\sim 1$ and $[k]< \gamma|a|$, we have $|\eta||a|\gtrsim 1/\gamma$. Hence
\begin{align}
\norm{ (H_k^\eta)^{1/2} }_{L^{p'}(B_r(a))}
& \lesssim_{N} (|\eta| |a|^{2})^{-N} 
  \lesssim_{N,\gamma} |a|^{-N}  \label{eq:prf-rest-8}
\end{align}
for any $N\in\N$. Gathering the estimates \cref{eq:prf-rest-hermite}, \cref{eq:hermite-kernel-L2}, \cref{eq:prf-rest-8} yields
\begin{equation}\label{eq:beta}
\beta_{k,\eta} \lesssim_{N,\gamma} [k]^{d_1/2} |a|^{-N}.
\end{equation}
Furthermore, recall that Minkowski's integral inequality and the Stein--Tomas restriction estimate gave us \cref{eq:prf-rest-5}, which yields in particular
\begin{equation} \label{eq:prf-rest-7}
\norm{\norm{g_k^\eta}_{L^p(\R^{d_1})}}_{L^2(\R^{d_2}_{\eta})}
 \lesssim \norm{F}_2 \norm{f}_p.
\end{equation}
Altogether, \cref{eq:prf-rest-hoelder}, \cref{eq:beta} and \cref{eq:prf-rest-7} provide
\begin{align*}
\norm{ P_k^\eta g_k^\eta }_{L^2(\R^{d_1}\times\R_{\eta}^{d_2})}
& \le \big \Vert \beta_{k,\eta} \norm{g_k^\eta}_{L^p(\R^{d_1})}  \big\Vert_{L^2(\R_\eta^{d_2})}  \\
& \lesssim_{N,\gamma} [k]^{d_1/2} |a|^{-N} \norm{F}_2 \norm{f}_p.
\end{align*}
Finally, by choosing $N\in\N$ large enough, we obtain
\begin{align*}
\sum_{[k]< \gamma |a|} \norm{ P_k^\eta g_k^\eta }^2_{L^2(\R^{d_1}\times\R_{\eta}^{d_2})}
& \lesssim_{N,\gamma} \sum_{[k]< \gamma |a|} [k]^{d_1} |a|^{-2N} \norm{F}_2^2 \norm{f}_p^2 \\
& \lesssim_\gamma |a|^{-2d_2(1/p- 1/2)} \norm{F}_2^2 \norm{f}_p^2.
\end{align*}
This finishes the proof.
\end{proof}

\section{Proofs of \texorpdfstring{\cref{thm:Lp-multiplier}}{Theorem 1.1} and \texorpdfstring{\cref{thm:riesz-means}}{Theorem 1.2}}
\label{sec:proof}

Let again $\varrho$ denote the Carnot--Carathéodory distance associated to the Grushin operator $L$, let $d=d_1+d_2$ be the topological dimension, and $Q=d_1+2d_2$ be the homogeneous dimension of the metric measure space $(\R^d,\varrho,|\cdot|)$. Moreover, let $p_{d_1,d_2}$ be defined as in \cref{eq:p}. Given any bounded Borel function $G:\R\to \C$, let
\[
G^{(j)} := (\hat G \chi_j)^\vee \quad \text{for } j\in\Z,
\]
where $\chi_j$ is defined by \cref{eq:dyadic}.

We will use the following result of \cite[Proposition I.22]{ChOuSiYa16}, which we record here in a slightly modified version, see the remark below. The proof of the result in \cite{ChOuSiYa16} relies on standard Calderón-Zygmund theory arguments.

\begin{proposition}\label{prop:weak-multiplier}
Let $L$ be a non-negative self-adjoint operator on a metric measure space $(X,d,\mu)$ of homogeneous type with homogeneous dimension $Q$. Let $1\le p_0 < p <2$. Suppose that $L$ satisfies the following properties:
\begin{enumerate}
    \item $L$ satisfies the finite propagation speed property.
    \item For all $t>0$ and all bounded Borel functions $F:\R\to \C$ supported in $[0,1]$,
    \begin{equation}\label{eq:stein-tomas-cond}
        \norm{F(t\sqrt L) \chi_{B_R} }_{p_0\to 2} \le C_{p_0} \Big( \frac{(R/t)^Q}{\mu(B_R)} \Big)^{1/p_0-1/2} \norm{F}_{\infty} .
    \end{equation}
    for all balls $B_R\subseteq X$ of radius $R>t$.
\end{enumerate}
Then for any $s>1/2$ and every bounded Borel function $F:\R\to \C$ satisfying $\norm{F}_{\sloc,s}<\infty$ and
\begin{equation}\label{eq:cond-1}
\norm{(F\chi_i)^{(j)}(\sqrt L)}_{p\to p} \le \alpha(i+j) \norm{F}_{\sloc,s} \quad\text{for all } i,j\in \Z,
\end{equation}
with $\sum_{\iota \ge 1}\iota \alpha(\iota)\le C_{p,s}$, the operator $F(\sqrt L)$ is bounded on $L^p$, and
\begin{equation}\label{eq:weak-type-estimate}
\norm{F(\sqrt L)}_{p\to p} \le C_{p,s} \norm{F}_{\sloc,s}.
\end{equation}
\end{proposition}

\begin{remark}
Proposition~I.22 of \cite{ChOuSiYa16} requires the condition $(\mathrm E_{p_0,2})$ in place of the Stein--Tomas restriction type condition \cref{eq:stein-tomas-cond}, which is however an equivalent property by Proposition~I.3 of the same paper. The additionally required condition (I.3.12) in \cite{ChOuSiYa16} is automatically fulfilled by Theorem I.5. Furthermore, in \cite{ChOuSiYa16} it is only stated that the operator $F(\sqrt L)$ is of weak type $(p,p)$, but $L^p$-boundedness can easily be recovered via interpolation, while the estimate \cref{eq:weak-type-estimate} follows by the closed graph theorem. The assumption $s>1/2$ in \cref{prop:weak-multiplier} ensures that $\norm{F}_\infty \lesssim \norm{F}_{\sloc,s}$.
\end{remark}

With \cref{prop:weak-multiplier} at hand, the proofs of \cref{thm:Lp-multiplier} and \cref{thm:riesz-means} boil down to proving the following statement.

\begin{proposition}\label{prop:dyadic-mult}
Let $1\le p\le p_{d_1,d_2}$ and $G:\R\to \C$ be an even bounded Borel function supported in $[-2,-1/2]\cup[1/2,2]$ such that $G\in L^2_s(\R)$ for some $s>d( 1/p - 1/2)$. Then there exists $\varepsilon>0$ such that
\[
\Vert G^{(\iota)}(\sqrt L) \Vert_{p\to p} \le C_{p,s} 2^{-\varepsilon\iota} \norm{G^{(\iota)}}_{L^2_s} \quad\text{for all } \iota\ge 0.
\]
\end{proposition}

Before we prove \cref{prop:dyadic-mult}, we briefly show how \cref{thm:Lp-multiplier} and \cref{thm:riesz-means} follow. The Bochner--Riesz summability of \cref{thm:riesz-means} (for $p>1$) might be seen as a consequence of \cref{thm:Lp-multiplier}, but it is however a direct consequence of \cref{prop:dyadic-mult}, without any Calderón-Zygmund theory involved.

\begin{proof}[Proof of \cref{thm:riesz-means}]
Let $G(\lambda):=(1-\lambda^2)_+^\delta$. As in \cref{prop:cc-distance} (4), define the dilations $\delta_t$ via $\delta_t (x,y) := (tx,t^2 y)$ for $t>0$ and $(x,y)\in\R^{d_1}\times \R^{d_2}$. Since $L$ is homogenous with respect to $\delta_t$, we have
\[
G(\sqrt L)(f \circ \delta_t)
 = (G(t\sqrt{L})f)\circ \delta_t.
\]
Hence
\[
\norm{(1-t^2L)_+^\delta}_{p\to p} = \norm{(1-L)_+^\delta}_{p\to p} \quad\text{for all } t>0.
\]
Thus we may assume $t=1$. Choose $s>0$ such that $d( 1/p - 1/2)<s<\delta+1/2$. Let $J_\alpha$ be the Bessel function of the first kind of order $\alpha>-1/2$, i.e.,
\[
J_\alpha(r) = \frac{(\lambda / 2)^\alpha}{\Gamma(\alpha+ 1/2) \pi^{1 / 2}} \int_{-1}^{1} e^{i r \lambda} (1-\lambda^2)^{\alpha-1 / 2} \,d\lambda,
\quad r>0.
\]
Since $|J_\alpha(r)| \lesssim r^{-1/2}$ (see Lemma 3.11 in Chapter IV of \cite{StWe71} for instance),
\[
|\hat G(\xi)| \sim |\xi|^{-\delta-1/2} |J_{\delta+1/2}(|\xi|)| \lesssim |\xi|^{-\delta-1} \quad \text{for } \xi\in\R\setminus\{0\}.
\]
Hence $|\xi^s \hat G(\xi)| \lesssim |\xi|^{s-\delta-1}$ and therefore $G\in L_s^2(\R)$ since $s-\delta-1 < -1/2$. We may decompose $G=G\psi + G(1-\psi)$ where $\psi:\R\to\C$ is a bump function supported in $[-3/4,3/4]$ with $\psi(\lambda)=1$ for $|\lambda|\le 1/2$. Then $G\psi$ is a bump function that may be treated for instance by the Mikhlin--Hörmander type result of \cite[Theorem 1]{MaMue14a}. Moreover, applying \cref{prop:dyadic-mult} for $G(1-\psi)$, we obtain
\[
\norm{(G(1-\psi))^{(\iota)}(\sqrt L)}_{p\to p} \lesssim 2^{-\varepsilon\iota} \norm{G}_{L_s^2(\R)} \quad \text{for }\iota\ge 0.
\]
Furthermore, $\sum_{\iota<0} (G(1-\psi))^{(\iota)}=(G(1-\psi))*(\sum_{\iota<0}\chi_\iota)^\vee$ is a Schwartz function that may again be treated by Theorem 1 of \cite{MaMue14a}. Taking the sum over all $\iota\ge 0$ finishes the proof.
\end{proof}
\begin{proof}[Proof of \cref{thm:Lp-multiplier}]
Since $\norm{F}_{\sloc,s}\sim \norm{\tilde F}_{\sloc,s}$ where $F(\lambda)=\tilde F(\sqrt \lambda)$, we may replace $F(L)$ by $F(\sqrt L)$ in the proof. Moreover, we may assume without loss of generality that $F$ is an even function since $L$ is a positive operator. To show $L^p$-boundedness of $F(\sqrt L)$, we verify the assumptions of \cref{prop:weak-multiplier}. Note that $s>1/2$ since $p\le p_{d_1,d_2}$. The required condition \cref{eq:stein-tomas-cond} is a consequence of \cref{eq:rest-2} and \cref{eq:rest-3}. Indeed, in our setting, since $|B_R(a,b)| \sim R^d \max\{R,|a|\}^{d_2}$ by \cref{prop:cc-distance}~(2), the first factor of the right hand site of \cref{eq:stein-tomas-cond} is given by
\begin{align*}
\Big(\frac{(R/t)^Q}{|B_R(a,b)|}\Big)^{1/p_0-1/2}
 & \sim t^{-Q(1/p_0-1/2)} \quad \text{if }R\ge |a|/4,
\end{align*}
and, since $R>t$,
\begin{align*}
\Big(\frac{(R/t)^Q}{|B_R(a,b)|}\Big)^{1/p_0-1/2}
 & \sim (|a|^{-d_2} t^{-d} (R/t)^{d_2})^{1/p_0-1/2}  \\
 & \ge (|a|^{d_2}t^d)^{-(1/p_0-1/2)}  \quad \text{if }R<|a|/4.
\end{align*}
Let $\delta_t$ be again the dilation from \cref{prop:cc-distance} (4). Then
\begin{equation}\label{eq:homo}
F(\sqrt L)(f \circ \delta_t)
 = (F(t\sqrt{L})f)\circ \delta_t.
\end{equation}
Let $t>0$ and $F$ be supported in $[1/2,2]$. Since $\varrho$ is homogeneous with respect to $\delta_t$ by \cref{prop:cc-distance} (4), \cref{eq:rest-3} yields for $R<|a|/4$
\begin{align}
\norm{ F(t\sqrt L) (\chi_{B_R^\varrho(a,b)}f)}_2
 & = t^{Q/2} \norm{ F(\sqrt L)(\chi_{B_{R/t}^\varrho(a/t,b/t^2)} (f\circ \delta_t))}_2 \notag \\
 & \lesssim t^{Q/2} (|a|/t)^{-d_2(1/ p_0 - 1/ 2)} \norm{ F}_2 \norm{f\circ \delta_t}_{p_0} \notag \\
 & = (t^d |a|^{d_2})^{-( 1/ p_0 - 1/ 2)} \norm{ F}_2 \norm{f}_{p_0}. \label{eq:weight-rest-4}
\end{align}
Given a bounded Borel function $F:\R\to\C$ supported in $[0,1]$, we decompose $F$ as
\[
F = \sum_{i\le 1} F\chi_i.
\]
Applying \cref{eq:weight-rest-4} for $\tilde t=t/2^i$ and $\tilde F=F(2^i\,\cdot\,) \chi$ and using $\norm{\tilde F}_2\lesssim \norm{F}_\infty$, we obtain
\begin{align*}
\norm{F(t\sqrt L) f}_2
& \le \sum_{i\le 1} \Vert (F\chi_i)(t\sqrt L)f \Vert_2 \\
& \lesssim \sum_{i\le 1} ((t/2^i)^d |a|^{d_2})^{-( 1/ p_0 - 1/ 2)} \norm{F}_\infty \norm{f}_{p_0} \\
& \sim (t^d |a|^{d_2})^{-( 1/ p_0 - 1/ 2)} \norm{F}_\infty \norm{f}_{p_0}.
\end{align*}
The computation for the case $R\ge|a|/4$ is similar. This establishes condition \cref{eq:stein-tomas-cond}.

Now we verify \cref{eq:cond-1}. For $i\in\Z$, let $F_i:=  F \chi_i$. Given $i,j\in \Z$, let $\iota:=i+j$ and
\[
G(\lambda):=F(2^i\lambda)\chi(\lambda),\quad\lambda\in\R,
\]
where $\chi$ is given by \cref{eq:dyadic}. Then $G$ is an even function, and
\begin{align*}
(F_i)^{(j)}(\lambda)
& = (\widehat{F_i} \chi_j)^\vee(\lambda)
 = (2^i\hat G(2^i\cdot) \chi_j)^\vee(\lambda) \notag \\
& = (\hat G \chi_\iota)^\vee(2^{-i}\lambda)
 = G^{(\iota)}(2^{-i}\lambda).
\end{align*}
Moreover, by the homogeneity \cref{eq:homo},
\[
\norm{G^{(\iota)}(2^{-i}\sqrt L)}_{p\to p} = \norm{G^{(\iota)}(\sqrt L)}_{p\to p}.
\]
Hence, for $\iota \ge 0$, \cref{prop:dyadic-mult} provides
\begin{align*}
\norm{(F\chi_i)^{(j)}(\sqrt L)}_{ p\to  p}
& 
 = \norm{G^{(\iota)}(\sqrt L)}_{ p\to  p} \notag \\
& \lesssim 2^{-\varepsilon\iota} \Vert G^{(\iota)}\Vert_{L_s^2}
  \lesssim 2^{-\varepsilon\iota} \norm{F}_{\sloc,s}.
\end{align*}
The case $\iota<0$ will be treated by the Mikhlin--Hörmander type result of \cite{MaMue14a}. Suppose $\iota<0$. Let $\psi:=\sum_{i\le 2} \chi_i$. Then $\psi$ is supported in $[-8,8]$. We decompose $G^{(\iota)}$ as $G^{(\iota)}=G^{(\iota)}\psi + G^{(\iota)}(1-\psi)$. Since $G^{(\iota)}=G*\check \chi_\iota$, $\supp G\subseteq [-2,2]$ and $\chi\in\S(\R)$, we have
\begin{align}
\Big|\Big(\frac{d}{d\lambda}\Big)^\alpha  G^{(\iota)} (\lambda )\Big|
 & = \Big|\Big(\frac{d}{d\lambda}\Big)^\alpha \int_{-2}^2 2^\iota G(\tau) \check \chi(2^\iota(\lambda-\tau))\, d\tau  \Big| \notag \\
 & \lesssim_N 2^{\iota(\alpha+1)} \int_{-2}^2 \frac{|G(\tau)|}{(1+2^\iota|\lambda-\tau|)^N}\, d\tau,
 \quad \alpha\in\N. \label{eq:error-conv}
\end{align}
Choosing $N:=0$ in \cref{eq:error-conv} and using $2^{\iota(\alpha+1)}\le 1$, we obtain
\begin{equation}\label{eq:sloc-small}
\norm{G^{(\iota)}\psi}_{\sloc,\lceil d/2 \rceil}\lesssim_{\psi} \norm{G}_2 \lesssim \norm{F}_{\sloc,s}.
\end{equation}
On the other hand, choosing $N:=\alpha+1$ in \cref{eq:error-conv} yields in particular
\[
\Big|\Big(\frac{d}{d\lambda}\Big)^\alpha  G^{(\iota)} (\lambda )\Big| \lesssim |\lambda|^{-\alpha} \norm{G}_2 \quad \text{for } |\lambda|\ge 4.
\]
Since all derivatives of $1-\psi$ are Schwartz functions, Leibniz rule yields
\begin{equation}\label{eq:sloc-wide-away}
\norm{G^{(\iota)}(1-\psi)}_{\sloc,\lceil d/2 \rceil}\lesssim_{\psi} \norm{G}_2 \lesssim \norm{F}_{\sloc,s}.
\end{equation}
Hence applying Theorem 1 of \cite{MaMue14a} provides
\begin{align*}
\norm{(F_i)^{(j)}(\sqrt L)}_{p\to p}
& = \norm{G^{(\iota)}(2^{-j}\sqrt L)}_{p\to p} \\
& = \norm{G^{(\iota)}(\sqrt L)}_{p\to p} \lesssim \norm{F}_{\sloc,s}.
\end{align*}
This establishes \cref{eq:cond-1}. Hence we may apply \cref{prop:weak-multiplier}.
\end{proof}

The rest of this section is devoted to the proof of \cref{prop:dyadic-mult}. The approach of our proof is essentially the same as in the proofs of Lemma~4.1 and Theorem~4.2 in \cite{ChOu16}. The new feature is the decomposition into eigenvalues of the rescaled Hermite operator $L^\eta$ via the truncation along the spectrum of $T$ afforded by the operators $\chi_\ell(L/T)$. This truncation corresponds to a subtler analysis of the sub-Riemannian geometry regarding the finite propagation speed property. A central ingredient of this analysis is the following weighted Plancherel estimate from \cite[Lemma~11]{MaMue14a}, which we can fortunately use out of the box.

\begin{lemma}\label{lem:weight-center}
Let $H:\R\to \C$ be a bounded Borel function supported in $[1/8,8]$, and, for $\ell\in\N$, let $H_\ell:\R\times \R \to \C$ be defined by
\[
H_\ell(\lambda,r) = H(\sqrt\lambda) \chi_\ell(\lambda/r) \quad\text{for }r\neq 0
\]
and $H_\ell(\lambda,r)=0$ else.
Then, for all $N\in\N$ and almost all $(a,b)\in \R^{d_1}\times\R^{d_2}$,
\begin{align*}
\int_{\R^d} \big| |y-b|^N & \, \mathcal K_{H_\ell(L,T)}((x,y),(a,b)) \big|^2 \,d(x,y) \le C_{\chi,N} 2^{\ell(2N - d_2)} \norm{H}_{L^2_N}^2,
\end{align*}
where $\mathcal K_{H_\ell(L,T)}$ denotes the integral kernel of the operator $H_\ell(L,T)$.
\end{lemma}

\begin{proof}[Proof of \cref{prop:dyadic-mult}]
Let $\iota\in\N$ and $R:=2^\iota$. We proceed in several steps.

\smallskip

(1) \textit{Reduction to compactly supported functions.} Let $f\in D(\R^d)$. We will first show that we may restrict to functions supported in balls of radius $R$ with respect to the Carnot--Carathéodory distance $\varrho$. Recall that $\varrho$ induces the Euclidean topology on $\R^d$, which implies in particular that the metric space $(\R^d,\varrho)$ is separable.
Since the metric measure space $(\R^d,\varrho,|\cdot|)$ is a space of homogeneous type, we may thus choose a decomposition into disjoint sets $B_n \subseteq B_R^\varrho(a_n,b_n)$, $n\in\N$, such that for every $\lambda\ge 1$, the number of overlapping dilated balls $B_{\lambda R}^\varrho(a_n,b_n)$ may be bounded by a constant $C(\lambda)$, which is independent of $\iota$. We decompose $f$ as
\[
f = \sum_{n=0}^\infty f_n\quad \text{where } f_n:=f\chi_{B_n}.
\]
Since $G$ is even, so is $\hat G$. As $\chi_\iota$ is even as well, the Fourier inversion formula provides
\[
G^{(\iota)} ( \sqrt L)f_n = \frac{1}{2\pi} \int_{2^{\iota-1}\le |\tau|\le 2^{\iota+1}} \chi_\iota(\tau) \hat G(\tau) \cos(\tau \sqrt L)f_n \,d\tau.
\]
By \cref{prop:cc-distance} (5), $L$ satisfies the finite propagation speed property, whence $G^{(\iota)} ( \sqrt L)f_n$ is supported in $B_{3R}^\varrho(a_n,b_n)$ by the formula above. Since the balls $B_{3R}^\varrho(a_n,b_n)$ have only a bounded overlap, we obtain
\[
\norm{G^{(\iota)} ( \sqrt L)f_n}_p^p \lesssim \sum_{n=0}^\infty \Vert G^{(\iota)} ( \sqrt L)f_n\Vert_p^p.
\]
Thus, since the functions $f_n$ have disjoint support, it suffices to show
\begin{equation}\label{eq:prf-dyadic-g_n}
\Vert G^{(\iota)} ( \sqrt L)f_n \Vert_p \lesssim 2^{-\varepsilon\iota} \norm{G^{(\iota)}}_{L^2_s} \norm{f_n}_p,
\end{equation}
with a constant independent of $n\in\N$.

\smallskip

(2) \textit{Localizing the multiplier.} Next we show that only the part of the multiplier $G^{(\iota)}$ located at $|\lambda|\sim 1$ is relevant. Let $\psi:=\sum_{|i|\le 2} \chi_i$. Then $\psi$ is supported in $\{\lambda \in\R :  1/8\le |\lambda| \le 8 \}$, while $1-\psi$ is supported in $\{\lambda \in\R : |\lambda| \notin (1/4,4) \}$. We decompose $G^{(\iota)}$ as $G^{(\iota)}=G^{(\iota)}\psi + G^{(\iota)}(1-\psi)$. The second part of this decomposition can be treated by the Mikhlin--Hörmander type result of \cite{MaMue14a}. As in \cref{eq:error-conv}, we observe
\begin{equation}\label{eq:error-conv-2}
\Big|\Big(\frac{d}{d\lambda}\Big)^\alpha  G^{(\iota)} (\lambda )\Big|
  \lesssim_N 2^{\iota(\alpha+1)} \int_{-2}^2 \frac{|G(\tau)|}{(1+2^\iota|\lambda-\tau|)^N}\, d\tau,
 \quad \alpha\in\N.
\end{equation}
Recall that $G$ is supported in $[-2,-1/2]\cup [1/2,2]$. Thus, choosing $N:=\alpha+2$ in \cref{eq:error-conv-2}, we obtain
\[
\Big|\Big(\frac{d}{d\lambda}\Big)^\alpha  G^{(\iota)} (\lambda )\Big| \lesssim  2^{-\iota} \min\{|\lambda|^{-\alpha},1\} \norm{G}_2 \quad \text{whenever } |\lambda|\notin [1/4,4].
\]
Similar as in \cref{eq:sloc-small} and \cref{eq:sloc-wide-away}, we obtain
\[
\norm{G^{(\iota)}(1-\psi)}_{\sloc,\lceil d/2 \rceil}\lesssim_{\psi} 2^{-\iota} \norm{G}_2.
\]
Hence applying Theorem 1 of \cite{MaMue14a} provides
\[
\Vert (G^{(\iota)}(1-\psi))(\sqrt L)\Vert_{p\to p}
 \lesssim 2^{-\iota} \norm{G}_2.
\]
Thus, in place of \cref{eq:prf-dyadic-g_n}, it suffices to show
\begin{equation}
\Vert \chi_{B_{3R}^\varrho(a_n,b_n)} (G^{(\iota)}\psi)(\sqrt L) f_n \Vert_p \lesssim 2^{-\varepsilon\iota} \norm{G^{(\iota)}}_{L^2_s} \norm{f_n}_p.\label{eq:prf-dyadic-g_n-tilde}
\end{equation}
To that end, we distinguish the cases $|a_n|>4R$ and $|a_n|\le 4R$.

\smallskip

(3) \textit{The elliptic region.} Suppose $|a_n|>4R$. Then, by \cref{prop:cc-distance} (2),
\[
|B_{3R}^\varrho(a_n,b_n)|\sim R^d \max\{R,|a_n|\}^{d_2} = R^d |a_n|^{d_2}.
\]
Let $2\le q\le \infty$ such that $1/q=1/p-1/2$. Applying Hölder's inequality together with the restriction type estimate \cref{eq:rest-3} for the multiplier $G^{(\iota)}\psi|_{[0,\infty)}$ (recall that $L$ is a positive operator) yields
\begin{align*}
\Vert \chi_{B_{3R}^\varrho(a_n,b_n)} (G^{(\iota)}\psi)(\sqrt L) f_n\Vert_p
 & \lesssim (R^d |a_n|^{d_2})^{1/q} \Vert (G^{(\iota)}\psi)(\sqrt L) f_n\Vert_2 \\
 & \lesssim 2^{\iota d/q}\norm{G^{(\iota)}\psi}_2\norm{f_n}_p \\
 & \lesssim 2^{-\varepsilon\iota} \norm{G^{(\iota)}}_{L^2_s}\norm{f_n}_p
\end{align*}
if we choose $0<\varepsilon<s-d/q$. This shows \cref{eq:prf-dyadic-g_n-tilde} in the case $|a_n|>4R$.

\smallskip

(4) \textit{The non-elliptic region: Truncation along the spectrum of $T$.} Suppose $|a_n|\le 4R$. Let $G_\ell^{(\iota)} : \R\times \R \to \C$ be given by
\[
G_\ell^{(\iota)}(\lambda,r) = (G^{(\iota)}\psi) (\sqrt \lambda)\chi_\ell(\lambda/r) \quad \text{for }r\neq 0
\]
and $G_\ell^{(\iota)}(\lambda,r) = 0$ else. We decompose the function on the left hand side of \cref{eq:prf-dyadic-g_n-tilde} as
\begin{align*}
& \chi_{B_{3R}^\varrho(a_n,b_n)} (G^{(\iota)}\psi)(\sqrt L) f_n \\
& = \chi_{B_{3R}^\varrho(a_n,b_n)} \bigg(\sum_{\ell =0}^\iota + \sum_{\ell = \iota + 1}^\infty \bigg) G_\ell^{(\iota)}(L,T) f_n
  =: g_{n,\le \iota} + g_{n,> \iota}.
\end{align*}
The second summand $g_{n,> \iota}$ can be directly treated by \cref{thm:rest}. Indeed, \cref{prop:cc-distance} (2), Hölder's inequality and the restriction type estimate \cref{eq:rest-2} imply
\begin{align*}
\Vert g_{n,> \iota}\Vert_p
 & \lesssim R^{Q/q} \Vert g_{n,> \iota}\Vert_2 \\
 & \lesssim 2^{\iota(Q-d_2)/q} \norm{G^{(\iota)}\psi}_2 \norm{f_n}_p \\
 & \lesssim 2^{-\varepsilon\iota}  \norm{G^{(\iota)}}_{L^2_s} \norm{f_n}_p
\end{align*}
if we choose $0<\varepsilon<s-d/q$. Hence we are done once we have shown
\begin{equation}\label{eq:case-2.2}
\Vert g_{n,\le \iota} \Vert_p
 \lesssim  2^{-\varepsilon\iota} \norm{G^{(\iota)}}_{L^2_s} \norm{f_n}_p.
\end{equation}

\smallskip

(5) \textit{(Almost) finite propagation speed on Euclidean scales in the non-elliptic region.} The key idea is as follows: Since $T=(-\Delta_y)^{1/2}$, we have
\[
(G^{(\iota)}\psi) (\sqrt L)\chi_{\{2k+d_1\}}(L/T) 
 = H_{\iota} \big([k]\sqrt{-\Delta_y}\big)\chi_{\{2k+d_1\}}(L/T),
\]
where $H_{\iota}(\lambda):=(G^{(\iota)}\psi) \big(\sqrt{\lambda}\big)$. Thus one might expect that the operator $G_\ell^{(\iota)}(L,T)$ behaves roughly like $H_\iota(2^\ell \sqrt{-\Delta_y})$ regarding the finite propagation property. Since $|a_n|\le 4R$, \cref{prop:cc-distance} (3) yields
\[
B_R^\varrho(a_n,b_n)
 \subseteq B_{R}^{|\,\cdot\,|}(a_n) \times B_{CR^2}^{|\,\cdot\,|}(b_n).
\]
Hence, for every $0\le \ell\le \iota$, we find a decomposition of $B_n\subseteq B_R^\varrho(a_n,b_n)$ such that
\[
B_n = \bigcup_{m=1}^{M_{n,\ell}} B_{n,m}^{(\ell)},
\]
where $B_{n,m}^{(\ell)} \subseteq B_{R}^{|\,\cdot\,|}(a_n) \times B_{CR_\ell}^{|\,\cdot\,|}(b_{n,m}^{(\ell)})$ with $R_\ell:= 2^\ell R$ are disjoint subsets, and
\[
|b_{n,m}^{(\ell)} - b_{n,m'}^{(\ell)}| > R_\ell/2 \quad \text{for } m\neq m'.
\]
The number of subsets in this decomposition is bounded by
\begin{equation}\label{eq:M_nl}
M_{n,\ell}\lesssim (R^2/R_\ell)^{d_2} = 2^{(\iota-\ell)d_2}.
\end{equation}
Moreover, given $\gamma>0$, the number $N_\gamma$ of overlapping balls
\[
\tilde B_{n,m}^{(\ell)}:=B_{3R}^{|\,\cdot\,|}(a_n) \times B_{2^{\gamma\iota+1} C R_\ell}^{|\,\cdot\,|}(b_{n,m}^{(\ell)}),\quad 1\le m\le M_{n,\ell}
\] 
can be bounded by $N_\gamma \lesssim_\iota 1$, where
\[
A \lesssim_\iota B
\]
means $A\le 2^{C(p,d_1,d_2)\iota\gamma}B$ for some constant $C(p,d_1,d_2)>0$ depending only on the parameters $p,d_1,d_2$. (The parameter $\gamma>0$ is necessary for having rapid decay for the negligible part of the propagation, see \cref{eq:error-L1}. This trick has also been used in a similar fashion in \cite{Mue89}.) We decompose $f_n$ as
\[
f_n = \sum_{m=1}^{M_{n,\ell}} f_{n,m}^{(\ell)}
\quad \text{where } f_{n,m}^{(\ell)}:=f_n \chi_{B_{n,m}^{(\ell)}}.
\]
In the next step, we show that the function
\[
g_{n,m}^{(\ell)}:= \chi_{B_{3R}^\varrho(a_n,b_n)} G_\ell^{(\iota)}(L,T) f_{n,m}^{(\ell)}
\]
is essentially supported in the ball $\tilde B_{n,m}^{(\ell)}$. Let $\chi_{n,m}^{(\ell)}$ denote the indicator function of $\tilde B_{n,m}^{(\ell)}$. We decompose $g_{n,\le \iota}$ as
\[
g_{n,\le \iota} = \sum_{\ell=0}^\iota \sum_{m=1}^{M_{n,\ell}} \tilde g_{n,m}^{(\ell)} + \sum_{\ell=0}^\iota \sum_{m=1}^{M_{n,\ell}} (1-\chi_{n,m}^{(\ell)}) g_{n,m}^{(\ell)},
\]
where $\tilde g_{n,m}^{(\ell)}:=\chi_{n,m}^{(\ell)} g_{n,m}^{(\ell)}$. The first summand represents the essential parts of the propagation, while the second one should be seen as an error term.

For the first summand, we observe that Hölder's inequality and the bounded overlapping property of the balls $\tilde B_{n,m}^{(\ell)}$ imply
\begin{equation}\label{eq:small-ell-main-1}
\bigg\Vert \sum_{\ell=0}^\iota \sum_{m=1}^{M_{n,\ell}}
\tilde g_{n,m}^{(\ell)}\bigg\Vert_p^p
 \le ((\iota+1) N_\gamma)^{p-1} \sum_{\ell=0}^\iota \sum_{m=1}^{M_{n,\ell}} \Vert\tilde g_{n,m}^{(\ell)}\Vert_p^p.
\end{equation}
Using Hölder's inequality together with \cref{eq:rest-1} yields
\begin{align}
\Vert\tilde g_{n,m}^{(\ell)}\Vert_p & \lesssim_\iota (R^{d_1} R_\ell^{d_2})^{1/q} \norm{g_{n,m}^{(\ell)}}_2 \notag \\
& = 2^{\iota d/q+\ell d_2/q}  \norm{g_{n,m}^{(\ell)}}_2 \notag \\
& \lesssim 2^{\iota d/q} \norm{G^{(\iota)}}_2 \norm{f_{n,m}^{(\ell)}}_p \notag \\
  & \sim  2^{\iota (d/q -s)}  \norm{G^{(\iota)}}_{L^2_s} \norm{f_{n,m}^{(\ell)}}_p. \label{eq:small-ell-main-4}
\end{align}
By \cref{eq:small-ell-main-1} and \cref{eq:small-ell-main-4}, we obtain
\[
\bigg\Vert \sum_{\ell=0}^\iota \sum_{m=1}^{M_{n,\ell}}
\tilde g_{n,m}^{(\ell)}\bigg\Vert_p^p
 \lesssim 2^{- \varepsilon \iota} \norm{G^{(\iota)}}_{L^2_s}^p \norm{f_{n}}_p^p
\]
for some $\varepsilon>0$ provided we choose $\gamma>0$ small enough before. As an upshot, to verify \cref{eq:case-2.2} it remains to show
\begin{equation}\label{eq:error-Lp}
\bigg\Vert \sum_{\ell=0}^\iota \sum_{m=1}^{M_{n,\ell}} (1-
\chi_{n,m}^{(\ell)}) g_{n,m}^{(\ell)}\bigg\Vert_p \lesssim 2^{-\varepsilon\iota}\norm{G^{(\iota)}}_{L^2_s} \norm{f_n}_p.
\end{equation}

(6) \textit{The negligible part of the propagation.} For showing \cref{eq:error-Lp}, we interpolate between $L^1$ and $L^2$ via the Riesz--Thorin interpolation theorem. The $L^2$-estimate is allowed to be quite rough, since the rapid decay in terms of $2^\iota$ derives from the $L^1$-estimate. For the $L^2$-estimate, we employ the Sobolev embedding
\[
\norm{G^{(\iota)}}_\infty \lesssim \norm{G^{(\iota)}}_{L_{1/2+\delta}^2} 
\sim 2^{\iota(1/2+\delta)} \norm{G^{(\iota)}}_{L^2},\quad\delta>0,
\]
which in conjunction with Hölder's inequality and \cref{eq:M_nl} provides
\begin{align}
\bigg\Vert \sum_{\ell=0}^\iota \sum_{m=1}^{M_{n,\ell}} (1-
\chi_{n,m}^{(\ell)}) g_{n,m}^{(\ell)}\bigg\Vert_2
 & \le \norm{G^{(\iota)}\psi}_\infty \sum_{\ell=0}^\iota \sum_{m=1}^{M_{n,\ell}} \Vert f_{n,m}^{(\ell)}\Vert_2 \notag \\
 & \le \norm{G^{(\iota)}\psi}_\infty  ((\iota+1)M_{n,\ell})^{1/2} \norm{f_{n}}_2 \notag\\
 & \lesssim 2^{\iota+d_2/2} \norm{G^{(\iota)}}_2 \norm{f_{n}}_2. \label{eq:error-L2}
\end{align}
The $L^1$-estimate is derived from an $L^\infty$ integral kernel estimate. Let $\mathcal K_\ell^{(\iota)}$ denote the integral kernel of $G_\ell^{(\iota)}(L,T)$. Then
\[
G_\ell^{(\iota)}(L,T) f_{n,m}^{(\ell)}(x,y) = \int_{\R^d} \mathcal K_\ell^{(\iota)}((x,y),(a,b)) f_{n,m}^{(\ell)}(a,b) \,d(a,b).
\]
For $b\in\R^{d_2}$, define the set
\[
B_{n}^{(b)} := \{ (x,y)\in B_{3R}^\varrho(a_n,b_n) : |y-b|\ge 2^{\gamma\iota}C R_\ell \}.
\]
Note that $(x,y)\in \supp((1- \chi_{n,m}^{(\ell)})\chi_{B_{3R}^\varrho(a_n,b_n)})$ and $(a,b)\in\supp f_{n,m}^{(\ell)}$ imply
\[
|y-b_{n,m}^{(\ell)}|\ge  2^{\gamma\iota+1}CR_\ell
\quad\text{and} \quad |b-b_{n,m}^{(\ell)}|< CR_\ell,
\]
and thus in particular $(x,y)\in B_{n}^{(b)}$. Hence
\begin{align}
& \bigg\Vert \sum_{\ell=0}^\iota \sum_{m=1}^{M_{n,\ell}} (1-
\chi_{n,m}^{(\ell)}) g_{n,m}^{(\ell)}\bigg\Vert_1 \notag \\
& \le \int_{\R^d} \sum_{\ell=0}^\iota \sum_{m=1}^{M_{n,\ell}} (1- \chi_{n,m}^{(\ell)}(x,y)) \chi_{B_{3R}^\varrho(a_n,b_n)}(x,y)  \notag \\
& \qquad \times \int_{\R^d} |\mathcal K_\ell^{(\iota)}((x,y),(a,b)) f_{n,m}^{(\ell)}(a,b)| \,d(a,b)\,d(x,y)  \notag \\
& \le \int_{B_R^\varrho(a_n,b_n)} \int_{B_{n}^{(b)}} \sum_{\ell=0}^\iota \sum_{m=1}^{M_{n,\ell}} |\mathcal K_\ell^{(\iota)}((x,y),(a,b)) f_{n,m}^{(\ell)}(a,b)| \,d(x,y)\,d(a,b) \notag  \\
& = \int_{B_R^\varrho(a_n,b_n)} \kappa_\gamma(a,b) |f_n(a,b)| \,d(a,b), \label{eq:error-L1-2}
\end{align}
where
\[
\kappa_\gamma(a,b) := \sum_{\ell=0}^\iota  \int_{B_{n}^{(b)}} |\mathcal K_\ell^{(\iota)}((x,y),(a,b)) | \,d(x,y).
\]
Given $N\in\N$, the Cauchy--Schwarz inequality yields
\begin{align*}
\kappa_\gamma(a,b)
& \lesssim \sum_{\ell=0}^\iota (2^{\gamma\iota}R_\ell)^{-N} \int_{B_{3R}^\varrho(a_n,b_n)} \big||y-b|^N \mathcal  K_\ell^{(\iota)}((x,y),(a,b)) \big| \,d(x,y) \\
& \le  \sum_{\ell=0}^\iota (2^{\gamma\iota}R_\ell)^{-N} |B_{3R}^\varrho(a_n,b_n)|^{1/2} \\
& \qquad \times \bigg( \int_{\R^d}  \big||y-b|^N \mathcal  K_\ell^{(\iota)}((x,y),(a,b)) \big|^2 \,d(x,y) \bigg)^{1/2}.
\end{align*}
Recall that $R_\ell = 2^{\iota+\ell}$ and $R=2^\iota$, and $|a_n|\le 4R$. By \cref{prop:cc-distance} (2), we have
\[
|B_{3R}^\varrho(a_n,b_n)| \sim R^Q = 2^{\iota Q}.
\]
Now, applying \cref{lem:weight-center} for $H=G^{(\iota)}\psi|_{[0,\infty)}$, and using the fact
\[
2^{-\iota N}\norm{G^{(\iota)}\psi|_{[0,\infty)}}_{L^2_N}\lesssim_{\psi} \norm{G^{(\iota)}}_2,
\]
we get
\begin{align*}
 \kappa_\gamma(a,b)
  & \lesssim  \sum_{\ell=0}^\iota 2^{-(\gamma\iota+\iota+\ell)N}
  2^{\iota Q/2}  2^{\ell(N - d_2/2)} \norm{G^{(\iota)}\psi|_{[0,\infty)}}_{L^2_N} \\
  & \lesssim 2^{-\gamma \iota  N} 2^{\iota Q/2} \norm{G^{(\iota)}}_2.
\end{align*}
Hence, plugging this estimate into \cref{eq:error-L1-2}, we obtain
\begin{equation}\label{eq:error-L1}
\bigg\Vert \sum_{\ell=0}^\iota \sum_{m=1}^{M_{n,\ell}} (1-
\chi_{n,m}^{(\ell)}) g_{n,m}^{(\ell)}\bigg\Vert_1
 \lesssim 2^{-\gamma\iota N}  2^{\iota Q/2}  \norm{G^{(\iota)}}_2 \norm{f_n}_1.
\end{equation}
Via \cref{eq:error-L2} and \cref{eq:error-L1}, the Riesz--Thorin interpolation theorem provides
\[
\bigg\Vert \sum_{\ell=0}^\iota \sum_{m=1}^{M_{n,\ell}} (1-
\chi_{n,m}^{(\ell)}) g_{n,m}^{(\ell)}\bigg\Vert_p
\lesssim (2^{-\gamma\iota N}  2^{\iota Q/2})^{1-\theta} (2^{\iota+d_2/2})^\theta \norm{G^{(\iota)}}_2 \norm{f_n}_p,
\]
where $\theta:=2(1-1/p)<1$. Choosing $N\in\N$ large enough yields \cref{eq:error-Lp}, whence we are done with the proof.
\end{proof}


\begin{thebibliography}{10}
\bibitem{AhCoMaMue20}
J.~Ahrens, M.~G. Cowling, A.~Martini, and D.~M\"{u}ller.
\newblock Quaternionic spherical harmonics and a sharp multiplier theorem on
  quaternionic spheres.
\newblock {\em Math. Z.}, 294(3-4):1659--1686, 2020.

\bibitem{BoGu11}
J.~Bourgain and L.~Guth.
\newblock Bounds on oscillatory integral operators based on multilinear
  estimates.
\newblock {\em Geom. Funct. Anal.}, 21(6):1239--1295, 2011.

\bibitem{CaSj72}
L.~Carleson and P.~Sj\"{o}lin.
\newblock Oscillatory integrals and a multiplier problem for the disc.
\newblock {\em Studia Math.}, 44:287--299. (errata insert), 1972.

\bibitem{CaCoMaSi17}
V.~Casarino, M.~G. Cowling, A.~Martini, and A.~Sikora.
\newblock Spectral multipliers for the {K}ohn {L}aplacian on forms on the
  sphere in {$\mathbb C^n$}.
\newblock {\em J. Geom. Anal.}, 27(4):3302--3338, 2017.

\bibitem{ChOu16}
P.~Chen and E.~M. Ouhabaz.
\newblock Weighted restriction type estimates for {G}rushin operators and
  application to spectral multipliers and {B}ochner-{R}iesz summability.
\newblock {\em Math. Z.}, 282(3-4):663--678, 2016.

\bibitem{ChOuSiYa16}
P.~Chen, E.~M. Ouhabaz, A.~Sikora, and L.~Yan.
\newblock Restriction estimates, sharp spectral multipliers and endpoint
  estimates for {B}ochner-{R}iesz means.
\newblock {\em J. Anal. Math.}, 129:219--283, 2016.

\bibitem{Ch91}
M.~Christ.
\newblock {$L^p$} bounds for spectral multipliers on nilpotent groups.
\newblock {\em Trans. Amer. Math. Soc.}, 328(1):73--81, 1991.

\bibitem{CoKlSi11}
M.~G. Cowling, O.~Klima, and A.~Sikora.
\newblock Spectral multipliers for the {K}ohn sublaplacian on the sphere in
  {$\mathbb C^n$}.
\newblock {\em Trans. Amer. Math. Soc.}, 363(2):611--631, 2011.

\bibitem{CoSi01}
M.~G. Cowling and A.~Sikora.
\newblock A spectral multiplier theorem for a sublaplacian on
  {$\mathrm{SU}(2)$}.
\newblock {\em Math. Z.}, 238(1):1--36, 2001.

\bibitem{Gu16}
L.~Guth.
\newblock A restriction estimate using polynomial partitioning.
\newblock {\em J. Amer. Math. Soc.}, 29(2):371--413, 2016.

\bibitem{He93}
W.~Hebisch.
\newblock Multiplier theorem on generalized {H}eisenberg groups.
\newblock {\em Colloq. Math.}, 65(2):231--239, 1993.

\bibitem{Hoe60}
L.~H{\"{o}}rmander.
\newblock Estimates for translation invariant operators in {$L^{p}$} spaces.
\newblock {\em Acta Math.}, 104:93--140, 1960.

\bibitem{Hoe67}
L.~H{\"{o}}rmander.
\newblock Hypoelliptic second order differential equations.
\newblock {\em Acta Math.}, 119:147--171, 1967.

\bibitem{JeSa87}
D.~Jerison and A.~S\'{a}nchez-Calle.
\newblock Subelliptic, second order differential operators.
\newblock In {\em Complex analysis {III} ({C}ollege {P}ark, {M}d., 1985--86)},
  volume 1277 of {\em Lecture Notes in Math.}, pages 46--77. Springer, Berlin,
  1987.

\bibitem{Ka95}
G.~E. Karadzhov.
\newblock Riesz summability of multiple {H}ermite series in {$L^p$} spaces.
\newblock {\em Math. Z.}, 219(1):107--118, 1995.

\bibitem{KoTa05}
H.~Koch and D.~Tataru.
\newblock {$L^p$} eigenfunction bounds for the {H}ermite operator.
\newblock {\em Duke Math. J.}, 128(2):369--392, 2005.

\bibitem{LeRoSe14}
S.~Lee, K.~M. Rogers, and A.~Seeger.
\newblock Square functions and maximal operators associated with radial
  {F}ourier multipliers.
\newblock In {\em Advances in analysis. {T}he legacy of {E}lias {M}. {S}tein},
  volume~50 of {\em Princeton Mathematical Series}, pages 273--302. Princeton
  University Press, Princeton, NJ, 2014.

\bibitem{Ma15}
A.~Martini.
\newblock Spectral multipliers on {H}eisenberg-{R}eiter and related groups.
\newblock {\em Ann. Mat. Pura Appl. (4)}, 194(4):1135--1155, 2015.

\bibitem{MaMue14a}
A.~Martini and D.~M\"{u}ller.
\newblock A sharp multiplier theorem for {G}rushin operators in arbitrary
  dimensions.
\newblock {\em Rev. Mat. Iberoam.}, 30(4):1265--1280, 2014.

\bibitem{MaMue14b}
A.~Martini and D.~M\"{u}ller.
\newblock Spectral multiplier theorems of {E}uclidean type on new classes of
  two-step stratified groups.
\newblock {\em Proc. Lond. Math. Soc. (3)}, 109(5):1229--1263, 2014.

\bibitem{MaMueNi19}
A.~{Martini}, D.~{M{\"u}ller}, and S.~{Nicolussi Golo}.
\newblock Spectral multipliers and wave equation for sub-laplacians: lower
  regularity bounds of euclidean type.
\newblock {\em J. Eur. Math. Soc.}, 2022.
\newblock Published online first.
\newblock \url{https://doi.org/10.4171/JEMS/1191}.

\bibitem{MaSi12}
A.~Martini and A.~Sikora.
\newblock Weighted {P}lancherel estimates and sharp spectral multipliers for
  the {G}rushin operators.
\newblock {\em Math. Res. Lett.}, 19(5):1075--1088, 2012.

\bibitem{MaMe90}
G.~Mauceri and S.~Meda.
\newblock Vector-valued multipliers on stratified groups.
\newblock {\em Rev. Mat. Iberoam.}, 6(3-4):141--154, 1990.

\bibitem{Me84}
R.~Melrose.
\newblock Propagation for the wave group of a positive subelliptic second-order
  differential operator.
\newblock In {\em Hyperbolic equations and related topics ({K}atata/{K}yoto,
  1984)}, pages 181--192. Academic Press, Boston, MA, 1986.

\bibitem{Mue89}
D.~M{\"{u}}ller.
\newblock On {R}iesz means of eigenfunction expansions for the
  {K}ohn-{L}aplacian.
\newblock {\em J. Reine Angew. Math.}, 401:113--121, 1989.

\bibitem{MueRiSt95}
D.~M{\"{u}}ller, F.~Ricci, and E.~M. Stein.
\newblock Marcinkiewicz multipliers and multi-parameter structure on
  {H}eisenberg (-type) groups. {I}.
\newblock {\em Invent. Math.}, 119(2):199--233, 1995.

\bibitem{MueRiSt96}
D.~M{\"{u}}ller, F.~Ricci, and E.~M. Stein.
\newblock Marcinkiewicz multipliers and multi-parameter structure on
  {H}eisenberg (-type) groups. {II}.
\newblock {\em Math. Z.}, 221(2):267--291, 1996.

\bibitem{MueSt94}
D.~M{\"{u}}ller and E.~M. Stein.
\newblock On spectral multipliers for {H}eisenberg and related groups.
\newblock {\em J. Math. Pures Appl. (9)}, 73(4):413--440, 1994.

\bibitem{RoSi08}
D.~W. Robinson and A.~Sikora.
\newblock Analysis of degenerate elliptic operators of {G}ru\v{s}in type.
\newblock {\em Math. Z.}, 260(3):475--508, 2008.

\bibitem{So08}
C.~D. Sogge.
\newblock Lectures on eigenfunctions of the {L}aplacian.
\newblock In {\em Topics in mathematical analysis}, volume~3 of {\em Ser. Anal.
  Appl. Comput.}, pages 337--360. World Sci. Publ., Hackensack, NJ, 2008.

\bibitem{StWe71}
E.~M. Stein and G.~Weiss.
\newblock {\em Introduction to {F}ourier analysis on {E}uclidean spaces},
  volume~32 of {\em Princeton Mathematical Series}.
\newblock Princeton University Press, Princeton, NJ, 1971.

\bibitem{Ta04}
T.~Tao.
\newblock Some recent progress on the restriction conjecture.
\newblock In {\em Fourier analysis and convexity}, Appl. Numer. Harmon. Anal.,
  pages 217--243. Birkh\"{a}user Boston, Boston, MA, 2004.

\bibitem{Th93}
S.~Thangavelu.
\newblock {\em Lectures on {H}ermite and {L}aguerre expansions}, volume~42 of
  {\em Mathematical Notes}.
\newblock Princeton University Press, Princeton, NJ, 1993.

\bibitem{To79}
P.~A. Tomas.
\newblock Restriction theorems for the {F}ourier transform.
\newblock In {\em Harmonic analysis in {E}uclidean spaces}, Proc. Sympos. Pure
  Math., vol. XXXV, Part 1, pages 111--114. Amer. Math. Soc., Providence, RI,
  1979.

\bibitem{VaSaCo92}
N.~T. Varopoulos, L.~Saloff-Coste, and T.~Coulhon.
\newblock {\em Analysis and geometry on groups}, volume 100 of {\em Cambridge
  Tracts in Mathematics}.
\newblock Cambridge University Press, Cambridge, 1992.

\end{thebibliography}
\end{document}